\documentclass[11pt,a4paper,twoside,reqno]{amsart}
\pdfoutput=1
\title{Coclosed $G_2$-structures on $\text{SU}(2)^2$-invariant cohomogeneity one manifolds}
\author{Izar Alonso}
\address[I. Alonso]{Department of Mathematics \\ Rutgers University \\Piscataway, NJ 08854\\ USA}
\email{izar.alonso@rutgers.edu}
\thanks{The author was supported by the EPSRC under grant No.\ 2271784 and the National Science Foundation under Grant No.\ DMS-1928930}

\usepackage[dutch, english]{babel}
\usepackage{fancyhdr,amsmath,amssymb,amsbsy,amsfonts,latexsym,float,color,amsthm,mathabx,hyperref}
\usepackage{graphicx,cite,enumerate,tikz-cd,cancel,anysize,tensor,wasysym,comment,mathtools,enumitem}

\numberwithin{equation}{section} 

\usepackage{scalerel,stackengine}
\stackMath
\newcommand{\R}{\mathbb{R}}
\newcommand{\C}{\mathbb{C}}

\newcommand{\iprod}{\mathbin{\raisebox{\depth}{\scalebox{1}[-2]{$\lnot$}}}}

\newtheorem{theorem}{Theorem}[section]
\theoremstyle{definition}
\newtheorem{defi}[theorem]{Definition}
\newtheorem{lemma}[theorem]{Lemma}
\newtheorem{ex}[theorem]{Example}

\newtheorem{prop}[theorem]{Proposition}
\newtheorem{cor}[theorem]{Corollary}
\newtheorem{remark}[theorem]{Remark}

\numberwithin{equation}{section}

\begin{document}
\begin{abstract}
We consider two different $\text{SU}(2)^2$-invariant cohomogeneity one manifolds, one non-compact $M=\mathbb{R}^4 \times S^3$ and one compact $M=S^4 \times S^3$, and study the existence of coclosed $\text{SU}(2)^2$-invariant $G_2$-structures constructed from half-flat $\text{SU}(3)$-structures.
For $\R^4 \times S^3$, we prove the existence of a family of coclosed (but not necessarily torsion-free) $G_2$-structures which is given by three smooth functions satisfying certain boundary conditions around the singular orbit and a non-zero parameter.
Moreover, any coclosed $G_2$-structure constructed from a half-flat $\text{SU}(3)$-structure is in this family.
For $S^4 \times S^3$, we prove that there are no $\text{SU}(2)^2$-invariant coclosed $G_2$-structures constructed from half-flat $\text{SU}(3)$-structures.
\end{abstract}

\maketitle

\section{Introduction}

A $G_2$-structure in a smooth seven-manifold $M$ is a smooth positive 3-form $\varphi$, and so defines a unique Riemannian metric $g_\varphi$ and an orientation $\mu_\varphi$ on $M$. 
A $G_2$-structure is called 
\textit{closed} if $d\varphi = 0$, and 
\textit{coclosed} if $d {*}_\varphi \varphi = 0$, where ${*}_\varphi$ is the Hodge star operator associated to $(g_\varphi, \mu_\varphi)$.
Seven-dimensional manifolds with a $G_2$-structure have been of great interest in differential geometry, since if a $G_2$-structure is both closed and coclosed, the metric $g_\varphi$ has holonomy contained in $G_2$. 
They have also attracted substantial interest in String Theory and M-theory. In heterotic or type II string theory they give rise to three-dimensional vacua, while in M-theory to four-dimensional $N = 1$ Minkowski vacua.
Coclosed $G_2$-structures exist on any oriented spin seven-manifold \cite{CrowNord15}. They appear in the context of the heterotic $G_2$ system, which is studied both by mathematicians and physicists.
When considering compactifications of heterotic string theory down to 3 dimensions,
the \textit{heterotic $G_2$ system} is a system for both geometric fields and gauge fields over a manifold with a $G_2$-structure. 
The two gauge fields are $G_2$-instantons and have their curvatures related by the so-called \textit{heterotic Bianchi identity}.
The heterotic $G_2$ system also requires that, after a conformal transformation, the $G_2$-structure is coclosed.
It appears in the broader context of the heterotic systems on manifolds with special geometry, such as the Hull--Strominger system \cite{[30]Strominger, Hull}.
The heterotic $G_2$ system was first studied in \cite{GN95, Gauntlett01, FI02, FI03, Gauntlett04a, Gauntlett04b, Ivanov05, Lukas11, GLL12}.
For a full description of the system, we refer the reader for example to \cite{dlOLM21}.

The main aim of this paper is to find coclosed $G_2$-structures in the cohomogeneity one setting.
A \textit{cohomogeneity one manifold} is a Riemannian manifold with an action by isometries of a compact Lie group $G$ having a generic orbit of codimension one.
We are interested in cohomogeneity one manifolds as this is the next step of complexity after Reidegeld studied homogeneous coclosed $G_2$-structures in \cite{Reidegeld09}.
In \cite{Cleyton01}, Cleyton and Swann studied coclosed $G_2$-structures on cohomogeneity one manifolds where the acting Lie group $G$ was simple. 
For that reason, we are interested in the situation when $G$ is not simple.
Assuming that the $G$-action is almost effective and that the manifold is simply connected leads us to consider $G=\text{SU}(2)^2$, up to finite quotients.
As torsion-free $G_2$ manifolds have been studied in this case (\cite{BB13} and see below for references on torsion-free $G_2$-structures on $\R^4 \times S^3$),
we will focus on the case when the $G_2$-structure is coclosed but not necessarily torsion-free.

In this paper, we consider two different $\text{SU}(2)^2$-invariant cohomogeneity one manifolds, one non-compact $M=\R^4 \times S^3$ and one compact $M=S^4 \times S^3$, and look for coclosed $\text{SU}(2)^2$-invariant $G_2$-structures constructed from half-flat $\text{SU}(3)$-structures, i.e.\ structures are the induced structures on hypersurfaces of a $G_2$-manifold.
By considering this type of $G_2$-structure, we are able to write them in a simplified way \cite{Schulte10, MS13}. This will allow us to write conditions for the structure to be coclosed as a system of linear ordinary differential equations.
We first find a large family of structures with these conditions on $\R^4 \times S^3$. The following Theorem summarizes the existence result (see Section \ref{sec:coclosedG2} for a more detailed statement).

\begin{theorem}\label{thm:C}
On the cohomogeneity one manifold $M=\R^4 \times S^3$ with group diagram $\text{SU}(2)^2 \supset \Delta \text{SU}(2) \supset \{ 1 \}$, there is a family of $\text{SU}(2)$-invariant coclosed $G_2$-structures which is given by three positive smooth functions satisfying certain boundary conditions around the singular orbit, and a non-zero parameter.
Moreover, any $\text{SU}(2)$-invariant coclosed $G_2$-structure constructed from a half flat $\text{SU}(3)$-structure is in this family.
\end{theorem}

On $S^4 \times S^3$, we prove that there are no $\text{SU}(2)^2$-invariant coclosed $G_2$-structures constructed from half-flat $\text{SU}(3)$-structures.

On the manifold $\R^4 \times S^3$, there are two explicit complete $G_2$-holonomy metrics: the Bryant--Salamon (BS) metric \cite{BS89} and the Brandhuber et al.\ (BGGG) metric \cite{Brandhuber01}. The first one is one of the three asymptotically conical examples constructed by Bryant and Salamon in 1989, while the second one is a member of a family of complete $\left( \text{SU}(2)^2 \times \text{U}(1) \right)$-invariant $G_2$-holonomy metrics found by Bogoyavlenskaya in \cite{Bogoyavlenskaya13}, known as the $\mathbb{B}_7$ family.
More recently, Foscolo, Haskins and Nordstr\"om constructed infinitely many new 1-parameter families of simply connected complete noncompact $G_2$–manifolds \cite{FHN21}, which are also $\left( \text{SU}(2)^2 \times \text{U}(1) \right)$-invariant. 
In \cite{Podesta19}, Podestà proved the existence of a one-parameter family of locally defined nearly parallel $G_2$-structures, which are mutually non-isomorphic and invariant under the cohomogeneity one action of the group $\text{SU}(2)^3$.

The family of $G_2$-structures from Theorem \ref{thm:C} includes a $\text{SU}(2)^3$-invariant subfamily of explicit structures.
It also includes some known families of torsion-free $G_2$-structures as particular cases: the Bryant--Salamon $G_2$-holonomy metric \cite{BS89} and the 1-parameter family of complete $\left( \text{SU}(2)^2 \times \text{U}(1) \right)$-invariant $G_2$-holonomy metrics of Brandhuber et al.\ \cite{Brandhuber01} and Bogoyavlenskaya \cite{Bogoyavlenskaya13}. 

\subsection*{Layout of the paper}
In Section \ref{sec:prelims}, we will start by giving a review of the geometry of $G_2$-structures, the associated splitting of differential forms into irreducible $G_2$ representations and the torsion classes of the structure in \ref{sec:G2}.
We also explain how to construct a $G_2$-structure from a half-flat $\text{SU}(3)$-structure in \ref{subsec:SU(3)}.
In Section \ref{sec:coh1}, we will develop the mathematical background on cohomogeneity one manifolds with two, one or zero ends.

We will start Section \ref{sec:SU(2)^2} by describing which seven-dimensional simply connected manifolds can be of cohomogeneity one for the almost effective action of a compact Lie group.
Then, in Section \ref{sec:coh1g2}, we give a description of a $G_2$-structure on a manifold which is invariant under the cohomogeneity one action of $\text{SU}(2)^2$. 
In Section \ref{sec:extension} we discuss the conditions for the metric to be extended smoothly to a singular orbit for the two manifolds of our interest. 
In Section \ref{sec:coclosedeqns} we obtain the systems of equations for the $G_2$-structure to be coclosed and write them as a system of linear ordinary differential equations. Then we consider the special case of $\text{SU}(2)^3$-invariant coclosed $G_2$-structures and give some examples in Section \ref{sec:examples}.

In Section \ref{sec:coclosedG2} we find a class of coclosed $G_2$-structures on a given seven-dimensional simply connected manifold under the cohomogeneity one action of $\text{SU}(2)^2$.
In Section \ref{sec:existence} we prove some existence results for the coclosed equations. We present and prove our results for the non-compact manifold $\R^4 \times S^3$ in Section \ref{sec:extnoncomp} and for the compact manifold $S^4 \times S^3$ in \ref{sec:extcomp}.
Finally, in Section \ref{sec:conclusions}, we make a few concluding remarks.

\subsection*{Acknowledgements}
I would like to thank my supervisors Andrew Dancer and Jason Lotay for their constant support, advice, and helpful exchange of ideas.
I also want to thank Christoph B\"ohm, Johannes Nordstr\"om and the anonymous reviewer for their helpful comments.
This material is based upon work carried out in the course of my DPhil studies at the University of Oxford and supported by the EPSRC under grant No.\ 2271784 and the National Science Foundation under Grant No.\ DMS-1928930 while the author was in residence at the Simons Laufer Mathematical Sciences Institute (previously known as MSRI) Berkeley, California, during the Fall 2022 semester, and accepted while the author in residence at the Simons Laufer Mathematical Sciences Institute during the Fall 2024 semester.

\section{Preliminaries}\label{sec:prelims}

\subsection{The geometry of $G_2$-structures}\label{sec:G2}

In this section, we will give a summary of the geometry of $G_2$-structures, from a Riemannian geometric point of view, including a discussion of the torsion. 
Through the past years, $G_2$-structures have been have been extensively studied. For more details, we refer the reader to \cite[Section 4]{Karigiannis19} and \cite[Chapter 11]{joyce}.

Consider $\R^7$ with the standard euclidean metric $g_0$, for which the standard basis $e_1, ..., e_7$ is orthonormal. Let $\mu_0 = e^1 \wedge ... \wedge e^7$ be the standard volume form associated to $g_0$ and the standard orientation.
Define the ``associative" 3-form $\varphi_0$ by
\begin{equation}\label{eq:varphi}
    \varphi_0= e^{123} -e^{167} -e^{527} -e^{563} -e^{415} -e^{426} -e^{437},
\end{equation}
where $e^1, ..., e^7$ is the standard dual basis of $(\R^7)^*$ and we write $e^{ijk}= e^i \wedge e^j \wedge e^k$, giving a different expression for the associative form.
Note that the order of $e^1, ..., e^7$ might change between references.

\begin{defi}
Let $M$ be a smooth seven-manifold. A \textit{$G_2$-structure} on $M$ is a smooth 3-form $\varphi$ on $M$ such that, at every $p \in M$, there exists a linear isomorphism $T_p M \cong \R^7$ with respect to which $\varphi_p \in \Lambda^3(T_p^* M)$ corresponds to $\varphi_0 \in \Lambda^3(\R^7)^*$. 
\end{defi}

A $G_2$-structure $\varphi$ on $M$ induces a Riemannian metric $g_\varphi$ and associated Riemannian volume form $\mu_\varphi$. Let $*_\varphi$ be the Hodge star operator induced from $(g_\varphi, \mu_\varphi)$. Then the \textit{coassociative form} is given by
$$
\psi= *_\varphi \varphi.
$$
A smooth seven-manifold admits a $G_2$-structure if and only if it is both orientable and spin \cite{Gray69}.

Let $\Omega^k= \Gamma(\Lambda^k(T^*M))$ be the space of smooth k-forms on $M$, and let $\Omega_l^k$ be the irreducible representation of $G_2$ of (pointwise) dimension $l$.
We can write a decomposition of $\Omega^2$ and $\Omega^3$ into irreducible $G_2$ representations:
$$
\Omega^2 = \Omega_7^2 \oplus \Omega_{14}^2,
$$
$$
\Omega^3= \Omega_1^3 \oplus \Omega_7^3 \oplus \Omega_{27}^3,
$$
where
\begin{equation}
\begin{array}{ll}
    \Omega_7^2&= \{ \beta \in \Omega^2 \ | \ {*} (\varphi \wedge \beta)=-2 \beta \}, \\
    \Omega_{14}^2&= \{ \beta \in  \Omega^2 \ | \ {*} (\varphi \wedge \beta)= \beta \}= \{ \beta \in \Omega^2 \ | \ \beta \wedge \psi=0 \} \cong \mathfrak{g}_2 , \\
    \Omega_1^3&= \{ f \varphi \ | \ f \in \Omega^0 \}, \\
    \Omega_7^3&= \{ X \iprod \psi \ | \ X \in \Gamma(TM) \}, \\
    \Omega_{27}^3&= \{ \gamma \in \Omega^3 \ | \ \gamma \wedge \varphi=0, \gamma \wedge \psi=0 \}.
\end{array}
\end{equation}

Let $M$ be a manifold with $G_2$-structure $\varphi$ and let $\nabla$ be the Levi-Civita covariant derivative associated to the Riemannian metric $g_\varphi$.

\begin{defi}
We can write 
$$
\nabla_X \varphi =T(X) \iprod \psi,
$$
for some vector field $T(X)$ on $M$. We call $T \in \Gamma(T^*M \otimes T^*M)$ the \textit{full torsion tensor} of $\varphi$.
\end{defi}

The following Proposition (\cite[Proposition 1]{Bryant05}, consequence of a classical result of Fern\'andez and Gray \cite{FG82}) gives a full description of the torsion of $\varphi$ is in terms of four quantities, known as \textit{intrinsic torsion forms}.

\begin{prop}\label{prop:torsionforms}
For any $G_2$-structure $\varphi$, there exist unique differential forms $\tau_0 \in \Omega^0(M)$, $\tau_1 \in \Omega^1$, $\tau_2 \in \Omega^2_{14}$ and $\tau_3 \in \Omega^3_{27}$ so that the following equations hold:
\begin{equation}\label{eq:dvarphi}
    d \varphi = \tau_0 \psi + 3 \tau_1 \wedge \varphi + * \tau_3,
\end{equation}
\begin{equation}\label{eq:dpsi}
    d \psi = 4 \tau_1 \wedge \psi + \tau_2 \wedge \varphi.
\end{equation}
\end{prop}

\begin{defi}
The $G_2$-structure $\varphi$ is called \textit{torsion-free} if $\nabla \varphi =0$.
We say $(M, \varphi)$ is a \textit{$G_2$-manifold} if $\varphi$ is a torsion-free $G_2$-structure on $M$.
\end{defi}
Hence, $\varphi$ is torsion-free if and only if $T=0$. Vanishing of the torsion is equivalent to the intrinsic torsion forms being all zero, so from (\ref{eq:dvarphi}) and (\ref{eq:dpsi}) we also have that $\varphi$ is torsion-free if and only if both $d \varphi=0$ and $d \psi=0$. 
When $\varphi$ is torsion-free, $g_\varphi$ is Ricci-flat and has holonomy contained in $G_2$.
Moreover, if $(M, \varphi, g_\varphi)$ is compact and $\varphi$ is torsion-free, then the holonomy $\text{Hol}(g_\varphi)= G_2$ if and only if $\pi_1(M)$ is finite (see \cite[Chapter 11]{joyce}).

We can classify $G_2$-structures in different types, depending on which components of the torsion vanish:
\begin{enumerate}[label=(\roman*)]
    \item Torsion-free ($\nabla \varphi=0$): $\tau_0 =\tau_1 =\tau_2 = \tau_3 =0$; 
    \item Closed ($d \varphi=0$): $\tau_0 =\tau_1 = \tau_3 =0$; 
    \item Coclosed ($d \psi=0$): $\tau_1 =\tau_2 =0$; 
    \item Nearly parallel ($d \varphi =\lambda \psi$, $\lambda \neq 0$): $\tau_1 =\tau_2 = \tau_3 =0$.
\end{enumerate}
In this paper, we will be interested in coclosed $G_2$-structures. 
The principal reason for this is that for a manifold to admit a solution to the heterotic $G_2$ system, it needs to admit a $G_2$-structure with intrinsic form $\tau_2=0$. After a conformal transformation, we can assume that $\tau_1=0$.
As proved by Crowley and Nordström in \cite{CrowNord15}, coclosed $G_2$-structures exist on any oriented spin seven-manifold. Coclosed structures have also been in the past
referred to as cosymplectic.

\subsection{Half-flat $\text{SU}(3)$-structures and $G_2$-structures}\label{subsec:SU(3)}

An \textit{$\text{SU}(3)$-structure} $(g, J, \Omega)$ on a six-dimensional smooth manifold $M$ is a Riemannian metric $g$ and a $g$-orthogonal almost complex structure $J$ together with a $(3,0)$-form of non-zero constant norm $\Omega=\Omega_1+ i\Omega_2$ satisfying the \textit{normalization condition}:
\begin{equation}
    \Omega_1 \wedge \Omega_2=\frac{2}{3}\omega^3.  
\end{equation}
We say that $\alpha \in \Omega^k(M)$ is \textit{stable} if for every $p \in M$, the orbit of $\alpha_p$ under the action of $\text{GL}(T_p M)$ is open in $\Omega^k (T_pM^*)$.
An $\text{SU}(3)$-structure on a six-dimensional smooth manifold can be described in terms of a pair of a 2-form $\omega$ and a 3-form $\Omega_2$ that are both stable (with $\lambda(\Omega_2) <0$) and such that they satisfy the \textit{compatibility condition}:
\begin{equation*}
    \omega \wedge \Omega_2=0,
\end{equation*}
provided that the corresponding metric is positive definite and the normalization condition is satisfied.
Given such forms, using Hitchin's construction \cite[Section 2]{[17]Hitchin]} we may recover the $\text{SU}(3)$-structure as $J=J_{\Omega_2}$ (where $J_{\Omega_2}$ is an almost complex structure depending only on $\Omega_2$), $g(\cdot, \cdot) \coloneqq \omega(\cdot, J \cdot)$ and $\Omega \coloneqq \Omega_1 + i \Omega_2$, where $\Omega_1 \coloneqq J \Omega_2$. For more details about this process, see for example \cite[Section 2.2]{AS22}.
We say that an $\text{SU}(3)$-structure given by the pair $(\omega, \Omega_2)$ is \textit{half-flat} if
$$
\text{d} \Omega_1=0, \quad \text{d} \omega^2 =0.
$$

We consider the 7-dimensional manifold $M= I \times N$, where $N$ is a smooth manifolds of dimension 6 and $I$ is an interval with coordinate $t \in \R$. Let $(\omega(t), \Omega_2(t))$ be a 1-parameter family of $\text{SU}(3)$-structures on $N$ parameterized by $t \in I$. Then, the following forms give a $G_2$-structure on $M$:
\begin{equation}\label{eq:G2 structure}
\begin{array}{ll}
\varphi = \text{d}t \wedge \omega(t) + \Omega_1(t), \\
\psi= \dfrac{\omega^2(t)}{2} - \text{d}t \wedge \Omega_2 (t).
\end{array}
\end{equation}
Every $G_2$-structure in the principal part of a cohomogeneity one manifold can be constructed from an $\text{SU}(3)$-structure in the principal orbits.
Assume that $(\omega(t), \Omega_2(t))$ is half-flat.
This $G_2$-structure is closed if the 1-parameter family $(\omega(t), \Omega_2(t))$ is a solution of:
$$
\dot{\Omega}_1 = \text{d} \omega.
$$
The $G_2$-structure is coclosed if $(\omega(t), \Omega_2(t))$ is a solution of:
$$
\omega \wedge \dot{\omega}= -\text{d} \Omega_2.
$$
Hence, the $G_2$-structure is torsion-free if $(\omega(t), \Omega_2(t))$ is a solution of both equations.

\begin{remark}
Note that a coclosed $G_2$-structure whose restriction to each principal orbit is half-flat only requires $d \Omega_1=0$ as an extra condition, as $d \omega^2 =0$ automatically holds.
\end{remark}

\subsection{Cohomogeneity one manifolds}\label{sec:coh1}
We will give a description of the structure of cohomogeneity one manifolds and their metrics.
For further details, see for instance \cite{Berand1982, [18]Hoelscher, [19]Hoelscher, Ziller07}.
\begin{defi}
A Riemannian manifold $M$ is of \textit{cohomogeneity one} for the action of the compact Lie group $G$ if
$G$ is a closed subgroup of the isometry group of $M$
and has an orbit of codimension one. 
\end{defi}

From now on, we assume that $M$ is a simply connected cohomogeneity one manifold, and $G$ is a compact, connected Lie group. 
By the compactness of $G$, the action is proper and there exists a $G$-invariant Riemannian metric $g$ on $M$; this is equivalent to saying that $G$ acts on the Riemannian manifold $\left(M,g\right)$ by isometries. 
Moreover, we assume that the cohomogeneity one action is \textit{almost effective}, i.e.\ that the set of points of $G$ whose action is trivial is discrete.
Let $\pi\colon M\to M/G$ be the canonical projection onto the orbit space and equip $M/G$ with the quotient topology relative to $\pi$.
The quotient space $M/G$ is then homeomorphic to a circle or an interval \cite{Berand1982}. As we are assuming that $M$ is simply connected, we have that $M/G$ is homeomorphic to an interval $I$.
The inverse images of the interior points of the orbit space $M/G$ are known as \emph{principal orbits}, while the inverse images of the boundary points are called \emph{singular orbits}. We denote by $M^{\text{princ}}$ the union of all principal orbits, which is an open dense subset of $M$, and by $G_p$ the isotropy group at $p\in M$.

First, we will suppose $M$ is compact. It follows that $M/G$ is homeomorphic to the closed interval $I=[0,1]$.
Denote by $\mathcal{O}_-$ and $\mathcal{O}_+$ the two singular orbits $\pi^{-1}\left(0\right)$ and $\pi^{-1}\left(1\right)$, respectively.
Let $\gamma\colon [0,1]\to M$ be a normal geodesic (i.e.\, geodesic orthogonal to every principal orbit) between $\pi^{-1}\left(0\right)$ and $\pi^{-1}\left(1\right)$, which always exists.
Up to rescaling, we can always suppose that the orbit space $M/G$ is such that $\pi \circ \gamma =\text{Id}_{[0,1]}$.
Then, by Kleiner's Lemma, there exists a subgroup $K$ of $G$ such that $G_{\gamma\left(t\right)}=K$ for all $t\in (0,1)$ and $K$ is subgroup of $G_{\gamma\left(0\right)}$ and $G_{\gamma\left(1\right)}$.

For $M$ non-compact, $M/G$ is homeomorphic either to an open interval or to an interval with one closed end.
In the former case, $M$ is a product manifold $M \cong I \times G/K$.
In the latter case, there exists exactly one singular orbit, and $M/G \cong I$ where $I=[0, \infty)$.
Analogously to the compact case, there exists a normal geodesic
$\gamma: [0, \infty) \rightarrow M$ such that $\gamma(0) \in \pi^{-1}(0)$ and we can suppose $\pi \circ \gamma =\text{Id}_{[0,\infty)}$. In addition, there exists a subgroup $K$ of $G$ such that $G_{\gamma(t)}=K$ for all $t \in (0,\infty)$ and $K$ is a subgroup of $G_{\gamma(0)}$. 

Up to conjugation along the orbits, when $M$ is compact we have three possible isotropy groups $H_-\coloneqq G_{\gamma\left(0\right)}$, 
$H_+ \coloneqq G_{\gamma\left(1\right)}$ and $K\coloneqq G_{\gamma\left(t\right)}$, $t\in\left(0,1\right)$. 
When $M$ is non-compact and has one singular orbit, instead,
we have two possible isotropy groups $H \coloneqq G_{\gamma\left(0\right)}$ and $K\coloneqq G_{\gamma\left(t\right)}$, $t\in\left(0,\infty \right)$.
From all of the above, we have that 
\[ M^{\text{princ}} \cong \overset{\circ}{I} \times G/K, 
\] and so, by fixing a suitable global coordinate system, we can decompose the $G$-invariant metric $g$ as
\begin{equation} \label{metric}
g_{\gamma(t)}={\text{d}}t^2 + g_t,
\end{equation}
where ${\text{d}}t^2$ is the $(0,2)$-tensor corresponding to the vector field $\xi \coloneqq \gamma'\left(t\right)$ evaluated at the point $\gamma\left(t\right)$, and $g_t$ is a $G$-invariant metric on the homogeneous orbit $G\cdot \gamma\left(t\right)$ through the point $\gamma\left(t\right)\in M$. 

Now, we will assume $M$ is compact. By the density of $M^{\text{princ}}$ in $M$ and the Tube Theorem, $M$ is homotopically equivalent to
\begin{equation}
\left( G \times_{H_-} D_-\right)\cup_{G/K}\left( G \times_{H_+} D_+\right),
\end{equation}
where the geodesic balls $ D_\pm \coloneqq \text{exp}\left(B_{\varepsilon^{\pm}}\left(0\right)\right)$, 
$B_{\varepsilon^{-}}\left(0\right)\subset T_{\gamma\left(0\right)}\left(G\cdot \gamma\left(0\right) \right)^{\perp}$, $B_{\varepsilon^{+}}\left(0\right)\subset T_{\gamma\left(1\right)}\left(G\cdot \gamma\left(1\right) \right)^{\perp}$, are normal slices to the singular orbits in $\gamma(0), \gamma(1)$. 
Here, $ G \times_{H_\pm} D_\pm$ is the associated fiber bundle to the principal bundle $G \to G/H_i$ with type fiber $D_\pm$.
By Bochner's linearization theorem, $M$ is also homotopically equivalent to 
\begin{equation}\label{DecompCohom1}
\left( G \times_{H_-} B_{\varepsilon^{-}}\left(0\right)\right)\cup_{G/K}\left( G \times_{H_+} B_{\varepsilon^{+}}\left(0\right)\right).
\end{equation}
The isotropy groups $H_\pm$ act on $ B_{\varepsilon^{\pm}}\left(0\right)$ via the slice representation and,
since the boundary of the tubular neighborhood $\text{Tub}(\mathcal{O}_\pm) \coloneqq  G \times_{H_\pm} B_{\varepsilon^{\pm}}\left(0\right)$,
is identified with the principal orbit $G/K$ and the $G$-action on 
$\text{Tub}(\mathcal{O}_\pm) $ is identified with the $H_\pm$-action
on $ B_{\varepsilon^{\pm}} \left( 0 \right)$, then $H_\pm$ acts transitively on
the sphere $S^{l_\pm} \coloneqq \partial  B_{\varepsilon^{\pm}}$, $l_\pm>0$ still having isotropy $K$. The normal spheres $S^{l_\pm}$ are thus the homogeneous
spaces $H_\pm/K$.

The collection of $G$ with its isotropy groups $G\supset H_-, H_+ \supset K $ is called the \emph{group diagram} of the cohomogeneity one manifold $M$.
Viceversa, let $G\supset H_-, H_+ \supset K $ be compact groups with $H_\pm/K \cong S^{l_\pm}$.
By the classification of transitive actions on spheres one has that the $H_\pm$-action on $S^{l_\pm}$ is linear and hence it can be extended
to an action on $ B_{\varepsilon^{\pm}}$ bounded by $S^{l_\pm}$.
Therefore, (\ref{DecompCohom1}) defines a cohomogeneity one manifold $M$.
Analogously, if $M$ is a non-compact cohomogeneity one manifold with one singular orbit, we define the \emph{group diagram} of $M$ to be the collection of $G$ and the isotropy groups $G \supset H \supset K$, where the homogeneous space $H/K$ will be a sphere. The converse is also true:  the group diagram defines a non-compact cohomogeneity one manifold $M$. In these cases, $M$ is homotopically equivalent to  $G\times_{H}B_{\epsilon}(0)$, where $B_{\epsilon}(0)\subseteq T_{\gamma(0)} (G\cdot \gamma(0))^{\perp}$ as before.

\section{$G_2$-structures on cohomogeneity one manifolds}\label{sec:SU(2)^2}

Let $M$ be a seven-dimensional simply connected manifold of cohomogeneity one for the almost effective action of a compact connected Lie group $G$. Let $\varphi$ be a $G_2$-structure on $M$ which is preserved by the action, and denote $\psi= *_\varphi \varphi$. Let $K$ be the principal isotropy group.

We require that the action of $G$ preserves the $G_2$-structure, so the principal isotropy group $K$ acts on $T_pM$ with $K \subset G_2$, for any principal point $p$ in $M$. If we denote $\mathfrak{k}=\text{Lie}(K)$, then $\mathfrak{k} \subset \mathfrak{g_2}$.
Let $\xi =\partial / \partial t$, where $t$ is the cohomogeneity one parameter.
The subgroup of $G_2$ that fixes the subspace $\langle \xi |_p \rangle$ of $T_p M$ is $\text{SU}(3)$.
Hence, the requirement that $G$ acts on $M$ with cohomogeneity one preserving $\varphi$ implies that the representation of the isotropy group $K=K_p$ on the tangent space of the principal orbit is a subgroup of $\text{SU}(3)$.
Therefore at the Lie algebra level $ \mathfrak{k}\coloneqq \text{Lie}\left(K\right)$ is $\{0\}, \, \mathbb{R}$, $\mathfrak{su}\left(2\right)$, $2\R$, $\mathfrak{u} =\mathfrak{su}(2) \oplus \R$ or $\mathfrak{su}(3)$. 
Note that $\mathfrak{su}(2)$ has two different embeddings in $\mathfrak{su}(3)$.

As $\dim \mathfrak{g}- \dim \mathfrak{k}=6$, we can write an initial list of possible possible decompositions of the Lie algebra $\mathfrak{g}$ into simple summands. 
This list can be reduced if we assume that the manifold $M$ is simply connected, using \cite[Proposition 1.8]{[18]Hoelscher} for the compact case and \cite[Proposition 3.1]{AS22} for the non-compact case.
If $\mathfrak{k} \subset \mathfrak{g}$ is an ideal, then the principal isotropy would act trivially on $M$ and the action would not be almost effective. 
Then, up to finite quotients, the principal orbits are one of the following types:
\begin{enumerate}
    \item $S^3 \times S^3 =\text{SU}(2)^2 = \dfrac{\text{SU}(2)^2 \times \text{U}(1)}{\text{U}(1)} =\dfrac{\text{SU}(2)^3}{\text{SU}(2)}$;
    \item
    $S^5 \times S^1 =\dfrac{\text{SU}(3) \times \text{U}(1)}{\text{SU}(2)}$;
    \item 
    $F_{1,2}=\dfrac{\text{SU(3)}}{T^2}$ ($\mathfrak{k}=2\R$, $\mathfrak{g}=\mathfrak{su}\left(3\right)$);
    \item
    $\C P(3)= \dfrac{\text{Sp}(2)}{\text{SU}(2) \text{U}(1)}$ ($\mathfrak{k}=\mathfrak{su}\left(2\right) \oplus \R$, $\mathfrak{g}=\mathfrak{sp}(2)$);
    \item
    $S^6=\dfrac{G_2}{\text{SU}(3)}$ ($\mathfrak{k}=\mathfrak{su}\left(3\right)$, $\mathfrak{g}=\mathfrak{g}_2$).
\end{enumerate}
For the last three cases, and up to a finite quotient, the group $G$ acting is $\text{SU}(3)$, $\text{Sp}(2)$ and $G_2$, respectively.
In \cite{Cleyton01}, Cleyton and Swann studied $G_2$-structures with a cohomogeneity one action of a compact Lie group $G$.
They only consider the case when the group $G$ acting is simple, so they consider acting groups $\text{SU}(3)$, $\text{Sp}(2)$ and $G_2$.
Later, they studied the coclosed (or cosymplectic, as they call them) structures and determined the topological types of manifolds admitting such structures. They also found new examples of compact manifolds with coclosed $G_2$-structures.
On account of this, our focus will be on the first case, where there is an invariant action of $\text{SU}(2)^2$. We will also consider situations where there are extra symmetries, such as extra $\text{SU}(2)$.

\subsection{$G_2$-structures on $\text{SU}(2)^2$-invariant cohomogeneity one manifolds}\label{sec:coh1g2}

In this section we construct an $\text{SU}(2)^2$-invariant half-flat $\text{SU}(3)$-structure on the principal part $\overset{\circ}{I_t} \times N$ 
of a cohomogeneity one manifold $M$, where $I_t$ is either $[0, \infty)$ or $[0,1]$, and $N=G/K$. 
We will describe the $\text{SU}(3)$-structure, the associated $G_2$-structure given by equation (\ref{eq:G2 structure}) and the corresponding metric on $M$ in terms of six real valued functions, by using an orthonormal basis given by a Milnor frame. 
This follows from the work by Schulte-Hegensbach in his PhD thesis \cite{Schulte10} and Madsen and Salamon in \cite{MS13}.
Here we will follow the reformulation by Lotay and Oliveira from \cite[Section 2]{Lotay17}.
We note that the induced $\text{SU}(3)$-structures $(\omega(t), \Omega_2(t))$ on the hypersurfaces of the manifold $N$ do not need to be half-flat (this is guaranteed when $M$ is a $G_2$-manifold, but when the $G_2$-structure is coclosed only the condition $d \omega^2=0$ is guaranteed). We make this assumption in order to use the description below to facilitate the search for new examples coclosed $G_2$-structures.

We can construct a basis of $\mathfrak{su}(2) \oplus \mathfrak{su}(2)$ written as
$
\mathfrak{su}(2) \oplus \mathfrak{su}(2) = \mathfrak{su}^+(2) \oplus \mathfrak{su}^-(2)
$.
Let $\{ T_i \}_{i=1}^3$ be a basis for $\mathfrak{su}(2)$ such that $[T_i, T_j] = 2 \epsilon_{ijk} T_k$. Then
$$
T_i^+=(T_i, T_i), \quad T_i^-=(T_i, -T_i),
$$
define a basis for $\mathfrak{su}^+(2)$ and $\mathfrak{su}^-(2)$ respectively.
Let $\{ \eta_i^+ \}_{i=1}^3$ and $\{ \eta_i^- \}_{i=1}^3$ be dual basis to $\{ T_i^+ \}_{i=1}^3$ and $\{ T_i^- \}_{i=1}^3$ respectively. 
Then the structure equations are
$$
\begin{array}{ll}
\text{d} \eta_i^+ = -\epsilon_{ijk} (\eta_j^+ \wedge \eta_k^+ + \eta_j^- \wedge \eta_k^-), \\
\text{d} \eta_i^- = -2\epsilon_{ijk} \eta_j^- \wedge \eta_k^+.
\end{array}
$$
We will denote $\eta_{ij}^\pm = \eta_i^\pm \wedge \eta_j^\pm$, and $\eta_{123}^\pm= \eta_1^\pm \wedge \eta_2^\pm \wedge \eta_3^\pm$.
A generic $\text{SU}(2)^2$-invariant half-flat $\text{SU}(3)$-structure on the principal bundle is given by
\begin{equation}\label{eq:SU(3)structure}
\begin{array}{ll}
\omega= 4 \sum_{i=1}^3 A_i B_i \eta_i^- \wedge \eta_i^+, \\
\Omega_1 = 8 B_1 B_2 B_3 \eta_{123}^- -4 \sum_{i=1}^3 \epsilon_{ijk} A_i A_j B_k  \eta_i^+ \wedge \eta_j^+ \wedge \eta_k^-, \\
\Omega_2 = -8 A_1 A_2 A_3 \eta_{123}^+ +4 \sum_{i=1}^3 \epsilon_{ijk} B_i B_j A_k  \eta_i^- \wedge \eta_j^- \wedge \eta_k^+,
\end{array}
\end{equation}
for six real-valued functions $A_i, B_i:I_t \rightarrow \R$, $i=1,2,3$, $A_i(t), B_i(t) \neq 0$ for $t$ in the interior of $I_t$.
The compatible metric determined by this $\text{SU}(3)$-structure on $\{ t \} \times N$ is:
$$
g_t =
\sum_{i=1}^3 (2A_i)^2 \eta_i^+ \otimes \eta_i^+ + \sum_{i=1}^3 (2B_i)^2 \eta_i^- \otimes \eta_i^-,
$$
and the resulting metric on $I_t \times N$, compatible with the $G_2$-structure $\varphi = \text{d}t \wedge \omega + \Omega_1$, is given by 
\begin{equation}\label{eq:g}
    g = \text{d}t^2 + g_t.
\end{equation}
Hence, we can see the functions $A_i(t)$ and $B_i(t)$ as describing deformations of the standard cone metric.

\begin{remark}
For a $G_2$-structure given by a half-flat $\text{SU}(3)$-structure as before, using the previous expressions and equation (\ref{eq:G2 structure}), we have that $\psi \iprod d \varphi = 0$, so
$$
\tau_0=0.
$$
This means that for all of the $G_2$-structures considered, the scalar curvature is zero. 
\end{remark}

\begin{remark}\label{rmk:flux}
In seven dimensional heterotic string theory, we say that the \textit{flux} is a 3-form given by
$$
H = \frac{1}{6} \tau_0 \varphi - \tau_1 \iprod \psi - \tau_3.
$$
One of the equations of the heterotic $G_2$ system, known as the \textit{heterotic Bianchi identity} or \textit{anomaly free condition}, relates the exterior differential of the flux to the curvatures of two gauge fields.
If the $G_2$-structure is coclosed as before, and using that $\tau_0=0$ from the previous remark, we get
$$
\begin{array}{ll}
dH
&=\frac{1}{6} \prod_{i=1}^3 | A_i B_i |  \\
&[16(-4 A_1 B_1 +(B_1 B_2 B_3)^\cdot -(A_2 A_3 B_1)^\cdot +(A_3 A_1 B_2)^\cdot +(A_1 A_2 B_3)^\cdot) \eta_{23}^- \wedge \eta_{23}^+ \\
&+16(-4 A_2 B_2 +(B_1 B_2 B_3)^\cdot +(A_2 A_3 B_1)^\cdot -(A_3 A_1 B_2)^\cdot +(A_1 A_2 B_3)^\cdot) \eta_{13}^- \wedge \eta_{13}^+ \\
&+16(-4 A_3 B_3 +(B_1 B_2 B_3)^\cdot +(A_2 A_3 B_1)^\cdot +(A_3 A_1 B_2)^\cdot -(A_1 A_2 B_3)^\cdot) \eta_{12}^- \wedge \eta_{12}^+]. \\
\end{array}
$$
This expression will be relevant when looking for solutions of the heterotic $G_2$ system for the $\text{SU}(2)^2$-invariant cohomogeneity one manifolds and the described coclosed $G_2$-structures.
This will be subject of future work.
\end{remark}

\subsection{Extension to the singular orbits}\label{sec:extension}

The union of the principal orbits $M^\text{princ}$ of the manifold $M$ is a dense subset of $M$. Hence, it is possible to extend the metric on the principal part to singular orbit(s) to give a metric on the manifold $M$. However, there are some extra conditions that we need to impose in order to ensure that this extension is smooth.
These conditions follow from a method developed by Eschenburg and Wang \cite{EW00} to find when a metric (or more generally, a tensor) extends smoothly to a singular orbit.
In \cite{verdiani2020smoothness}, Verdiani and Ziller gave an efficient way of checking the conditions for a smooth extension of the metric.
To use this method, we need to fix our cohomogeneity one manifold. We will consider two situations, one of a compact and one of a non-compact manifold.

As explained in Section \ref{sec:coh1}, a non-compact cohomogeneity one manifold is given by a homogeneous vector bundle, while a compact one by the union of two homogeneous disc bundles. 
We are interested in the smoothness conditions near a singular orbit, so we restrict ourselves to only one such bundle.

The non-compact manifold that we are going to consider is $M= \R^4 \times S^3$, seen as a cohomogeneity one manifold with group diagram $\text{SU}(2) \times \text{SU}(2) \supset \Delta \text{SU}(2) \supset \{ 1 \}$.
In particular, we consider the embedding $\R^4 \times S^3 \hookrightarrow \mathbb{H} \times \mathbb{H}$ and let $\text{SU}(2) \times \text{SU}(2)$ act via
$$
(a_1 , a_2) \cdot (p,q) = (a_1 p, a_1 q \bar{a}_2).
$$
$\R^4 \times S^3$ is diffeomorphic to the total space of the spinor bundle $\mathcal{S} \rightarrow S^3$ over the 3-sphere, which can be described as the quotient 
$$
\dfrac{\text{SU}(2) \times \text{SU}(2) \times \R^4}{\Delta \text{SU}(2)},
$$
where $\text{SU}(2)$ is acting on the right diagonally.
One of the reasons why we are interested in this manifold is that it admits torsion-free $G_2$-structures \cite{BS89, Brandhuber01, Bogoyavlenskaya13} and $G_2$-instantons \cite{Lotay17, Clarke14}.

The next Lemma tells us the conditions on the functions from the previous section $A_i, B_i$,
$i=1,2,3$, for the metric to extend to the singular orbit $Q= \text{SU}(2)^2/ \Delta \text{SU}(2) \cong S^3$.

\begin{lemma}\cite[Lemma 8]{Lotay17}\label{lemma8}
The metric $g$ in (\ref{eq:g}) extends smoothly (as a metric) over the singular orbit $Q= \text{SU}(2)^2/ \Delta \text{SU}(2)$ if and only if $A_i, B_i$ are non-zero for $t>0$ and:
\begin{enumerate}[label=(\roman*)]
    \item the $A_i$'s are odd with $\dot{A}_i(0)=1/2$;
    \item the $B_i$'s are even with $B_1(0)=B_2(0)=B_3(0) \neq 0$ and $\ddot{B}_1(0)=\ddot{B}_2(0)=\ddot{B}_3(0)$.
\end{enumerate}
\end{lemma}

Note that condition (i) says that the metrics have to be almost round near to the singular orbit, while condition (ii) guarantees that the singular orbit is totally geodesic.

\begin{remark}\label{rmk:inequivactions}
There is another inequivalent action of $\text{SU}(2)^2$ on $\R^4 \times S^3$, with group diagram $\text{SU}(2) \times \text{SU}(2) \supset \text{SU}(2) \times \{ 1 \} \supset \{ 1 \}$. 
Moreover, there are infinitely many distinct seven-dimensional cohomogeneity one manifold with one singular orbit and an inequivalent action of $\text{SU}(2)^2$. They have singular isotropy group $K_{m,n} = \{ (e^{i \theta_1}, e^{i \theta_2}) \in T^2 | e^{i(m \theta_1 +n \theta_2)} =1\}$, where $m,n$ are co-prime (this guarantees that the manifold is simply connected), and principal isotropy group isomorphic to $\mathbb{Z}_{2 |m+n|}$. The manifolds constructed from that action are denoted by $M_{m,n}$.
In \cite{FHN21}, Foscolo, Haskins and Nordstr\"om constructed 1-parameter families of torsion-free $G_2$-structures on these manifolds, known respectively as the $\mathbb{D}_7$ family and the infinite extension of the $\mathbb{C}_7$ family.
It would be interesting to study coclosed $G_2$-structures on those manifolds as well, for which the (different) conditions for a smooth extension to the singular orbit can be found in \cite[Proposition 4.1]{FHN21}. 
\end{remark}

For a compact manifold, we consider the cohomogeneity one manifold $M=S^4 \times S^3$ with group diagram $\text{SU}(2) \times \text{SU}(2) \supset \Delta \text{SU}(2), \Delta \text{SU}(2) \supset \{ 1 \}$.
We consider the embedding $S^3 \times S^4 \hookrightarrow \mathbb{H} \times (\mathbb{H} \times \R)$, then $S^3 \times S^3$ acts on $M$ via
$$
(a_1 , a_2) \cdot (p,q, t) = (a_1 p a_2^{-1}, a_2 q, t).
$$
The metric $g$ in (\ref{eq:g}) extends smoothly (as a metric) over the singular orbits $Q_{1,2}= \text{SU}(2)^2/ \Delta \text{SU}(2)$ if the conditions for $A_i, B_i$ from the previous lemma, which are now defined on the closed interval $[0,1]$, are satisfied around both singular orbits.
Note that in Lemma \ref{lemma8}, $t$ is the arclength parameter along a geodesic intersecting the principal orbits orthogonally, but we assume that the length of the interval $I_t$ is 1 instead of an arbitrary positive number, as it does not affect the discussion, and in order to be consistent with the notation of Section \ref{sec:coh1}.

\begin{remark}
For the two previous situations, the corresponding normal sphere(s) $H/K$ have dimension 3. It would be interesting to explore another example where there is a normal sphere with dimension 1, such as the Bazaikin-Bogoyavlenskaya manifolds \cite{BB13}, which have group diagram $\text{SU}(2)^2 \supset \text{U}(1) \supset \mathbb{Z}/4$.
\end{remark}

\subsection{Coclosed equations}\label{sec:coclosedeqns}

We derive equations for the functions $A_i$, $B_i$ such that the $G_2$-structure obtained via (\ref{eq:SU(3)structure}) is coclosed.
Recall from Section \ref{subsec:SU(3)} that they will be derived from
\begin{equation}\label{eq:coclosed}
\omega \wedge \dot{\omega} = -\text{d} \Omega_2.
\end{equation}
Using the expression for $\omega$ and $\Omega_2$ from equation (\ref{eq:SU(3)structure}), we obtain the system of ODEs that will give us the coclosed conditions. We present it in the following Proposition.

\begin{prop}
Let $M$ be a seven-dimensional simply connected cohomogeneity one manifold under the action of $\text{SU}(2)^2$, with a $G_2$-structure constructed from a half-flat $\text{SU}(3)$-structure. Let the $\text{SU}(3)$-structure $(\omega, \Omega_1, \Omega_2)$ be written as in (\ref{eq:SU(3)structure}).
Then, the equations for the $G_2$-structure to be coclosed are:
\begin{equation}\label{eq:coclosedAB}
\begin{cases}
A_2 B_2 (\dot{A}_3 B_3 + A_3 \dot{B}_3) +& A_3 B_3 (\dot{A}_2 B_2 + A_2 \dot{B}_2) =\\
&A_1 A_2 A_3 - B_2 B_3 A_1 + B_1 B_3 A_2 + B_1 B_2 A_3, \\
A_1 B_1 (\dot{A}_3 B_3 + A_3 \dot{B}_3) +& A_3 B_3 (\dot{A}_1 B_1 + A_1 \dot{B}_1) =\\
&A_1 A_2 A_3 + B_2 B_3 A_1 - B_1 B_3 A_2 + B_1 B_2 A_3,  \\
A_1 B_1 (\dot{A}_2 B_2 + A_2 \dot{B}_2) +& A_2 B_2 (\dot{A}_1 B_1 + A_1 \dot{B}_1) =\\
&A_1 A_2 A_3 + B_2 B_3 A_1 + B_1 B_3 A_2 - B_1 B_2 A_3.
\end{cases}
\end{equation}
\end{prop}

If we define
$$
\begin{array}{ll}
D_1 =A_2 B_2 A_3 B_3, \\
D_2 =A_1 B_1 A_3 B_3, \\
D_3 =A_1 B_1 A_2 B_2,
\end{array}
$$
then we can write (\ref{eq:coclosedAB}) as a system of ODEs for the functions $D_1, D_2, D_3$, where the coefficients depend on free functions $A_1, A_2, A_3$.
To ensure that the induced metrics are smooth, Lemma \ref{lemma8} leads us to consider $A_i$'s which are non-zero (positive, more precisely) for $t$ in the interior of $I_t$ and satisfy certain boundary conditions.
Hence, this system will be well defined in the interior of $I_t$ and we will study the behaviour for the boundary points, where the $A_i$'s are zero.
Note that the initial conditions from Lemma \ref{lemma8} correspond to the following conditions for the functions $D_i$
\begin{equation}\label{eq:Dinitial}
D_i(t)= \dfrac{b_0^2}{4} t^2 + O(t^4),
\end{equation}
and $D_i$ even, $i=1,2,3$.
The system is
\begin{equation}\label{eq:coclosedD}
\begin{cases}
\dot{D}_1 = A_1 A_2 A_3 - \dfrac{A_1^2 D_1}{A_1 A_2 A_3} + \dfrac{A_2^2 D_2}{A_1 A_2 A_3} + \dfrac{A_3^2 D_3}{A_1 A_2 A_3}, \\
\dot{D}_2 = A_1 A_2 A_3 + \dfrac{A_1^2 D_1}{A_1 A_2 A_3} - \dfrac{A_2^2 D_2}{A_1 A_2 A_3} + \dfrac{A_3^2 D_3}{A_1 A_2 A_3}, \\
\dot{D}_3 = A_1 A_2 A_3 + \dfrac{A_1^2 D_1}{A_1 A_2 A_3} + \dfrac{A_2^2 D_2}{A_1 A_2 A_3} - \dfrac{A_3^2 D_3}{A_1 A_2 A_3}.
\end{cases}
\end{equation}
Writing $D= (D_1, D_2, D_3)^T$ the system of equations now becomes
\begin{equation}\label{eq:coclosedDmatrix}
\dot{D}= 
\begin{pmatrix}
1 \\
1 \\
1
\end{pmatrix}
A_1 A_2 A_3 + \dfrac{1}{A_1 A_2 A_3} 
\begin{pmatrix}
-A_1^2 & A_2^2 & A_3^2 \\
A_1^2 & -A_2^2 & A_3^2 \\
A_1^2 & A_2^2 & -A_3^2
\end{pmatrix}
D.
\end{equation}
We will write it with matrix notation.
Let 
\begin{equation}\label{eq:MN}
M= \dfrac{1}{A_1 A_2 A_3}
\begin{pmatrix}
-A_1^2 & A_2^2 & A_3^2 \\
A_1^2 & -A_2^2 & A_3^2 \\
A_1^2 & A_2^2 & -A_3^2
\end{pmatrix}
, \quad 
N= A_1 A_2 A_3
\begin{pmatrix}
1 \\
1 \\
1 
\end{pmatrix}.
\end{equation}
Then our system is
\begin{equation}\label{eq:linearsysODE}
\dot{D}=M(t)D + N(t).
\end{equation}

\begin{remark}
We would like to recover the $B_i$'s from the $D_i$'s with:
\begin{equation}\label{eq:B_i}
\begin{array}{ll}
B_i(t) = \text{sign}(b_0) \sqrt{\dfrac{D_j D_k}{D_i A_i^2}}, t >0, \quad B_i(0)=b_0,
\end{array}
\end{equation}
where
$\ddot{D}_i(0) =b_0^2/2$
and $\{ i,j,k\}$ is a cyclic permutation of $\{ 1,2,3\}$. 
We can only do that if $D_i D_j /D_k >0$ for $t>0$.
\end{remark}

\subsection{$\text{SU}(2)^3$-invariant coclosed $G_2$-structures}\label{sec:examples}
We now study some special cases of coclosed $G_2$-structures, more concretely by considering the case where the $G_2$-structure enjoys an extra $\text{SU}(2)$-symmetry, i.e.\ $A_1 = A_2 = A_3$ and $B_1 = B_2 = B_3$. We will also study explicitly the case $A_1 = A_2 = A_3$ and show that the extra $\text{SU}(2)$-symmetry follows.

\subsubsection{Case $A_1=A_2=A_3$ and $B_1=B_2=B_3$.}\label{section:AAABBB}
We only have one equation:
$$
(A_1 B_1)^\cdot = \dfrac{1}{2} \dfrac{A_1^2 + B_1^2}{B_1}.
$$
If we define $D=A_1^2 B_1^2$, we get
\begin{equation}\label{eq:AAABBB}
    \dot{D} =  A_1^3+ \dfrac{D}{A_1}.
\end{equation}

\begin{ex}
If we take $A_1=t/2$, our equation is
$$
\dot{D} =  \dfrac{t^3}{8}+ \dfrac{2D}{t}.
$$
This has solutions depending on one constant, which we denote $c$:
$$
D= ct^2+ \dfrac{t^4}{16}.
$$
Hence renaming $4c$ by $c$:
$$
B_1= \sqrt{\dfrac{t^2}{4}+c}.
$$
This is an even function, and $B_1(0) \neq 0$ if and only if $c \neq 0$. For $c>0$, since it is well defined for $t \in [0, \infty)$, we can smoothly extend the metric to the singular orbit at $t=0$.
The corresponding metric is 
$$
g =dt^2 + t^2 \sum_{i=1}^3 \eta_i^+ \otimes \eta_i^+ + (t^2 +c) \sum_{i=1}^3 \eta_i^- \otimes \eta_i^-.
$$
In the limit $c \rightarrow 0$, the metric is conical.
\end{ex}

\begin{ex}
One of the first examples of a torsion-free $G_2$-structure with a complete metric is the one that gives the Bryant--Salamon metric on $\R^4 \times S^3$ from \cite{BS89}. The $G_2$-structure is then coclosed. 
The Bryant--Salamon metric is actually $\text{SU}(2)^3$-invariant, and it can be realised as a cohomogeneity one manifold with group diagram $\text{SU}(2)^3 \supset \text{SU}(2)^2 \supset \text{SU}(2)$ (where $\text{SU}(2)$ is embedded in $\text{SU}(2)^3$ as $1 \times 1 \times \text{SU}(2)$ and $\text{SU}(2)^2$ as $\Delta_{1,2} \text{SU}(2) \times \text{SU}(2)$), as well as with an action of $\text{SU}(2)^2$ in multiple inequivalent ways.
The extra symmetry means that $A_1 =A_2 =A_3$ and $B_1=B_2=B_3$, and the torsion-free equations reduce to 
$$
\dot{A_1}= \dfrac{1}{2} \left( 1- \dfrac{A_1^2}{B_1^2} \right), \qquad \dot{B_1}=\dfrac{A_1}{B_1}.
$$
There is a solution 
$$
A_1= \dfrac{r}{3} \sqrt{1-r^{-3}} , \qquad B_1 =\dfrac{r}{\sqrt{3}},
$$
where $r \in [1, + \infty)$ is a coordinate defined implicitly by
$$
t(r)= \int_1^r \dfrac{ds}{\sqrt{1-s^{-3}}},
$$
and $t$ denotes the arc length parameter.
When $t \rightarrow \infty$, $A_1(t) \sim t/3$ and $B_1(t) \sim t/\sqrt{3}$.
The metric will then be
$$
g=dt^2 + \sum_{i=1}^3  \left(  \dfrac{4 r^2}{9}(1 - r^{-3}) \eta_i^+ \otimes \eta_i^+ + \dfrac{4r^2}{3} \eta_i^- \otimes \eta_i^- \right).
$$
The asymptotic cone over of this metric is the standard homogeneous nearly K\"ahler structure on $S^3 \times S^3$.
\end{ex}

The general solution to equation (\ref{eq:AAABBB}) is 
\begin{equation}
    D=c e^{\int^t_{1/2} \frac{1}{A_1(\xi)}d\xi } + e^{\int^t_{1/2} \frac{1}{A_1(\xi)}d\xi} \int_0^t A_1^3(\eta) e^{- \int^\eta_{1/2} \frac{1}{A_1(\xi)} d\xi} d\eta.
\end{equation}
As $\lim_{t \rightarrow 0} D(t) t^2=4c$, and we are only interested in positive solutions, we rename $c$ by $b_0^2/16$ so that the behaviour of $D$ for small $t$ agrees with (\ref{eq:Dinitial}).
Let $f_{A_1}$ be the function defined by 
$$
f_{A_1}:t \mapsto e^{\int^t_{1/2} \frac{1}{A_1(\xi)}d\xi },
$$ 
we conclude that $D=\frac{b_0^2}{16} f_{A_1}(t) + f_{A_1}(t) \int_0^t A_1^3(\eta) f^{-1}_{A_1}(\eta) d\eta$ and
\begin{equation*}
B_1=\sqrt{A_1^{-2}(t) \left( \dfrac{b_0^2}{16} f_{A_1}(t) + f_{A_1}(t) \int_0^t A_1^3(\eta) f^{-1}_{A_1}(\eta) d\eta \right) }.
\end{equation*}
For these solutions, there is an extra $\text{SU}(2)$ symmetry, and hence they are not only $\text{SU}(2)^2$-invariant but actually $\text{SU}(2)^3$-invariant.

If we consider $\R^4 \times S^3$ and assume a linear asymptotic behaviour of $A_1$, i.e.\ 
$$
A_1(t) \sim at,
$$
$a>0$, when $t \rightarrow \infty$, we have linear asymptotic behaviour of $B_1$ (as in the Bryant--Salamon metric, for which $a=1/3$) where $a > 1/4$ (but only in this case). 
Specifically,
$$
B_1 (t) \sim \sqrt{\dfrac{a^2}{4a-1}} t.
$$

\subsubsection{Case $A_1=A_2=A_3$.}\label{subsec:AAA}

We will now see what happens when $A_1=A_2=A_3$.
Writing $D= (D_1, D_2, D_3)^T$ the system of equations (\ref{eq:coclosed}) is
\[ 
\dot{D}= 
\dfrac{1}{A_1} 
\begin{pmatrix}
-1 & 1 & 1 \\
1 & -1 & 1 \\
1 & 1 & -1
\end{pmatrix}
D +
\begin{pmatrix}
1 \\
1 \\
1
\end{pmatrix}
A_1^3
.
\]

We will solve the system for a generic smooth $A_1(t)$, with $A_1(t) \neq 0$ if $t \neq 0$.
By first solving the homogeneous system and then using variation of parameters for the the general solution to the in-homogeneous system, we find that the general solution is
\[
\begin{array}{ll}
D
=&
\begin{pmatrix}
-f^{-2}_{A_1}(t) & -f^{-2}_{A_1}(t) & f_{A_1}(t) \\
0 & -f^{-2}_{A_1}(t) & f_{A_1}(t) \\
-f^{-2}_{A_1}(t) & 0 & f_{A_1}(t)
\end{pmatrix}
\begin{pmatrix}
c_1  \\
c_2  \\
c_3 
\end{pmatrix}
+
\begin{pmatrix}
f_{A_1}(t) \int^t_0 A_1^3(\eta) f^{-1}_{A_1}(\eta) d \eta \\
f_{A_1}(t) \int^t_0 A_1^3(\eta) f^{-1}_{A_1}(\eta) d \eta \\
f_{A_1}(t) \int^t_0 A_1^3(\eta) f^{-1}_{A_1}(\eta) d \eta
\end{pmatrix},
\end{array}
\]
for some real constants $c_1, c_2, c_3$.
We observe that although $1/A_1(t)$ might not be locally integrable in a neighborhood of 0, the exponential of the
integral $\pm \int^t_{1/2} A_1^{-1}(\xi) d\xi$ need not present a singularity at 0 (see examples below). Similarly for \\
$f_{A_1}(t) \int_0^t A_1^3(\eta) f^{-1}_{A_1}(\eta) d\eta$.

As motivated from the conditions from Lemma \ref{lemma8}, we now suppose a Taylor expansion of $A_1(t)$ of the following form
$$
A_1(t)= \dfrac{t}{2} +a_{1,3} t^3 + a_{1,5} t^5 + ...
$$
We have then
$$
\int^t_{1/2} \dfrac{1}{A_1(\xi)} d\xi = \int^t_{1/2} \dfrac{2}{\xi} d\xi + \int^t_{1/2} \dfrac{a_{1,3} \xi + a_{1,5} \xi^3 + ...}{\frac{1}{4} + \frac{a_{1,3}}{2} \xi^2 + \frac{a_{1,5}}{2} \xi^4 + ...} d\xi
=2 \ln t + f(t).
$$
The second term is the integral of a smooth function, and we have denoted it by $f(t)$.
Hence
$$
f_{A_1}(t)  = e^{f(t)} t^2, \quad
f^{-2}_{A_1}(t)  = e^{-2f(t)} t^{-4}.
$$
Imposing $D_i(0)=0$ so that conditions from Lemma \ref{lemma8} can be satisfied, we see that we must have $c_1 =c_2 =0$, so $B_1 =B_2 =B_3$ and we are in special case of \ref{section:AAABBB}. Renaming $c_3$ as $b_0^2/16$ as before, the general solution becomes
\[
\begin{array}{ll}
D=
\begin{pmatrix}
\frac{b_0^2}{16} f_{A_1}(t)  + f_{A_1}(t)  \int_0^t A_1^3(\eta) f^{-1}_{A_1}(\eta)  d\eta \\
\frac{b_0^2}{16} f_{A_1}(t)  + f_{A_1}(t)  \int_0^t A_1^3(\eta) f^{-1}_{A_1}(\eta)  d\eta \\
\frac{b_0^2}{16} f_{A_1}(t)  + f_{A_1}(t)  \int_0^t A_1^3(\eta) f^{-1}_{A_1}(\eta)  d\eta \\
\end{pmatrix}.
\end{array}
\]

\section{Class of coclosed $G_2$-structures over a $\text{SU}(2)^2$-invariant manifold}\label{sec:coclosedG2}

\subsection{Existence and uniqueness results}\label{sec:existence}

The main tool that we are going to use to get existence and uniqueness results for coclosed $G_2$-structures is the following theorem (\cite[Theorem 7.1]{Mal74} and \cite{FH17} for this statement).

\begin{theorem}\cite[Theorem 4.7]{FH17}\label{thm:4.7}
Consider the singular initial value problem
\begin{equation}\label{eq:4.8}
\dot{y} = \dfrac{1}{t} M_{-1}(y) + M (t, y),
\quad
y(0) = y_0,
\end{equation}
where $y$ takes values in $\R^k$, $M_{-1}: \R^k \rightarrow \R^k$ is a smooth function of $y$ in a neighbourhood of $y_0$ and $M : \R \times \R^k \rightarrow \R^k$ is smooth in $t, y$ in a neighbourhood of $(0, y_0)$. Assume that
\begin{enumerate}[label=(\roman*)]
    \item $M_{-1} (y_0) = 0$;
    \item $h \text{Id} - d_{y_0} M_{-1}$ is invertible for all $h \in \mathbb{N}$, $h \geq 1$.
\end{enumerate}
Then there exists a unique solution $y(t)$ of (\ref{eq:4.8}). Furthermore $y$ depends continuously on $y_0$ satisfying (i) and (ii).
\end{theorem}

Note that this result only gives a short-time solution. However, we will be able to further extend the solution (see Remark \ref{rmk:domain}).

We will write our system as in this theorem. 
Let $A_1, A_2, A_3:[0, L) \rightarrow \R$ be smooth functions, where $L$ is either infinity or $1$, that satisfy the corresponding conditions from Lemma \ref{lemma8} (i), i.e.\ with $A_i(t) >0$ for $t \in (0,L)$, and such that
\begin{enumerate}[label=(\roman*)]
    \item $A_i$'s are odd;
    \item $\dot{A}_i(0)=1/2$.
\end{enumerate}
Let $D_0=(0, 0, 0)^T$, as $D_i(0)$ must be 0 if we want an extension to a singular orbit in $t=0$. Then we can write (\ref{eq:coclosedDmatrix})
as 
$$
\dot{D}=\dfrac{1}{t} M_{-1}(D) +M(t,D),
$$
where
\[
M_{-1}=2
\begin{pmatrix}
-1 & 1 & 1\\
1 & -1 & 1\\
1 & 1 & -1
\end{pmatrix}: \R^3 \rightarrow \R^3,
\]
and
\[
M(t, D)= A_1 A_2 A_3
\begin{pmatrix}
1\\
1\\
1
\end{pmatrix}
+ \dfrac{1}{A_1 A_2 A_3}
\begin{pmatrix}
-A_1^2 & A_2^2 & A_3^2 \\
A_1^2 & -A_2^2 & A_3^2 \\
A_1^2 & A_2^2 & -A_3^2
\end{pmatrix}
D
-\dfrac{1}{t} M_{-1} (D)
\]
is smooth in $t, D$.
The condition (i) from Theorem \ref{thm:4.7} holds as $M_{-1}(D_0)=0$. However, (ii) does not hold as $d_{y_0} M_{-1}$ has one positive eigenvalue: $d_{y_0} M_{-1} - h Id$ is not invertible for $h=2$:
\[
d_{y_0} M_{-1}=2
\begin{pmatrix}
-1 & 1 & 1\\
1 & -1 & 1\\
1 & 1 & -1
\end{pmatrix}.
\]

Dividing (\ref{eq:coclosedDmatrix}) by $t^2$, we get
$$
\dfrac{\dot{D}}{t^2}= \dfrac{1}{t}M_{-1} \left( \dfrac{D}{t^2} \right) + M' \left( t, \dfrac{D}{t^2} \right),
$$
where $M_{-1}$ is as before and
\[
M'(t, y)= \dfrac{1}{t^2} A_1 A_2 A_3
\begin{pmatrix}
1\\
1\\
1
\end{pmatrix}
+ \dfrac{1}{A_1 A_2 A_3}
\begin{pmatrix}
-A_1^2 & A_2^2 & A_3^2 \\
A_1^2 & -A_2^2 & A_3^2 \\
A_1^2 & A_2^2 & -A_3^2
\end{pmatrix}
y
-\dfrac{1}{t} 
2
\begin{pmatrix}
-1 & 1 & 1\\
1 & -1 & 1\\
1 & 1 & -1
\end{pmatrix}
y
\]
The function $M'(t, y)$ is smooth in $(t,y)$ in a neighborhood of $(0, D_0)$. Then
\[
\begin{array}{ll}
\dfrac{d}{dt} \left( \dfrac{D}{t^2} \right)
&=\dfrac{1}{t} (M_{-1} -2Id) \left( \dfrac{D}{t^2} \right) + M' \left( t, \dfrac{D}{t^2} \right) \\
&=\dfrac{1}{t} M_{-1}' \left( \dfrac{D}{t^2} \right) + M' \left( t, \dfrac{D}{t^2} \right),
\end{array}
\]
where 
\[
M_{-1}'= 2
\begin{pmatrix}
-2 & 1 & 1\\
1 & -2 & 1\\
1 & 1 & -2
\end{pmatrix}
: \R^3 \rightarrow \R^3.
\]
Now $d_{y_0}M_{-1}'$ only has non-negative eigenvalues. 
Fix \\
$E_0= (b_0^2/4, b_0^2/4, b_0^2/4)^T$ for some $b_0 \neq 0$. Note that $M_{-1}'(E_0)=0$.
The new singular initial value problem for $E(t)= D(t)/t^2$ is:
\begin{equation}\label{eq:IVPE}
\dot{E}=\dfrac{1}{t} M_{-1}'(E) + M'(t,E), \quad E(0)=E_0.
\end{equation}
It satisfies the conditions from Theorem \ref{thm:4.7}, meaning there exists a unique solution $E(t)$ in a neighbourhood of 0. Furthermore it depends continuously on $E_0$. Then we can define $D_1, D_2, D_3$ by $D(t)=t^2 E(t)$. 

\begin{remark}\label{rmk:domain}
As for every $i=1,2,3$ we have that $A_i$ is smooth and positive on $(0,L)$, the function $(D, t) \mapsto M(t) D +N(t)$ (with $M$ and $N$ are as in (\ref{eq:MN})) is Lipschitz in $D$ on any closed interval contained in $(0,L)$. Hence, using Picard--Lindel\"of Theorem we know that we can extend the solution in a neighbourhood of 0 to $[0, L)$.
\end{remark}

We have proved the next Proposition.

\begin{prop}\label{prop:exuniqsols}
Let $A_1, A_2, A_3:[0, L) \rightarrow \R$ be smooth functions with $A_i(t) >0$ for $t \in (0,L)$, where $L$ is either infinity or $1$,
such that
\begin{enumerate}[label=(\roman*)]
    \item $A_i$'s are odd;
    \item $\dot{A}_i(0)=1/2$.
\end{enumerate}
Let $b_0 \neq 0$. Consider the singular initial value problem
$$
\begin{array}{ll}
\dot{D}_1 = A_1 A_2 A_3 - \dfrac{A_1^2 D_1}{A_1 A_2 A_3} + \dfrac{A_2^2 D_2}{A_1 A_2 A_3} + \dfrac{A_3^2 D_3}{A_1 A_2 A_3}, \\
\dot{D}_2 = A_1 A_2 A_3 + \dfrac{A_1^2 D_1}{A_1 A_2 A_3} - \dfrac{A_2^2 D_2}{A_1 A_2 A_3} + \dfrac{A_3^2 D_3}{A_1 A_2 A_3}, \\
\dot{D}_3 = A_1 A_2 A_3 + \dfrac{A_1^2 D_1}{A_1 A_2 A_3} + \dfrac{A_2^2 D_2}{A_1 A_2 A_3} - \dfrac{A_3^2 D_3}{A_1 A_2 A_3},
\end{array}
$$
with
$$
D_i(0)=\dot{D}_i(0)=0, \quad \ddot{D}_i(0)=b_0^2/4.
$$
Then, there exists a unique solution on $[0,L)$ for this system of ODEs, that depends continuously on $b_0$.
\end{prop}

We need to check whether we can recover $B_1, B_2, B_3$ using equation (\ref{eq:B_i}) to obtain a coclosed $G_2$-structure in the principal part. We can only do that if $D_1 D_2 D_3 >0$.
The following Lemma guarantees that this is true.

\begin{lemma}\label{lemma:signDs}
Let $D_1, D_2, D_3: [0, L) \rightarrow \R$ be the unique solutions from Proposition \ref{prop:exuniqsols}, for some $L$ which is either a positive number or infinity. Then $D_i(t) >0$ for $t \in (0,L)$.
\end{lemma}
\begin{proof}
By construction $D_i(t) = \frac{b_0^2}{4} t^2 +O(t^3)$, so in a neighborhood around $0$ they are all positive. Suppose the statement is not true, and let $t'>0$ be the smallest positive real number where one of the $D_i$'s is $0$.
Then
$$
\dot{D}_i(t')=A_1(t')A_2(t')A_3(t')+\dfrac{A^2_j(t')D_j(t')}{A_1(t')A_2(t')A_3(t')}
+\dfrac{A^2_k(t')D_k(t')}{A_1(t')A_2(t')A_3(t')} > 0
$$
where $\{ i,j,k\}$ is a cyclic permutation of $\{ 1,2,3\}$. However, $D_i$ cannot be increasing at $t'$, so this is a contradiction.
\end{proof}

\begin{cor}\label{cor:exuniqsols}
Let $A_1, A_2, A_3:[0, L) \rightarrow \R$ be smooth functions with $A_i(t) >0$ for $t \in (0,L)$, where $L$ is either infinity or $1$, such that
\begin{enumerate}[label=(\roman*)]
    \item $A_i$'s are odd;
    \item $\dot{A}_i(0)=1/2$.
\end{enumerate}
Let $b_0 \neq 0$. Then there exist unique functions $B_1, B_2, B_3:[0,L) \rightarrow \R$ that, together with $A_1, A_2, A_3$, will give a solution for the system of ODEs (\ref{eq:coclosedAB}) with $B_i(0)=b_0$, and the $B_i$'s depend continuously on $b_0$.
\end{cor}

We could also study the system (\ref{eq:coclosedD}) without appealing to the theory of singular initial value problems. We may do so using the same procedure of Section \ref{subsec:AAA}, by first solving it using the general theory of linear systems of ordinary differential equations and then introducing the boundary conditions of the functions $A_1, A_2, A_3$.
We describe the main points but omit the details, as it follows the same steps as in Section \ref{subsec:AAA}.
If $\lambda_1(t)$, $\lambda_2(t)$, $\lambda_3(t)$ are the eigenvalues of $M(t)$, then 
the general solution will be equal to a particular solution plus a linear combination of $c_1 \exp{\int_{1/2}^t \lambda_1(\xi)}$, $c_2 \exp{\int_{1/2}^t \lambda_2(\xi)}$, and $c_3 \exp{\int_{1/2}^t \lambda_3(\xi)}$, where $c_1, c_2, c_3$ are arbitrary constants. Near the singular orbit and without loss of generality, we have $\lambda_1 = -4/t + O(t)$, $\lambda_2 = -4/t + O(t)$, and $\lambda_3 = 2/t + O(t)$, so imposing $D_i(0)=0$ we get $c_1=c_2=0$. Finally, $c_3$ will be uniquely determined by $b_0^2$ and the desired boundary conditions for the solution (\ref{eq:Dinitial}) also follow.
While this alternative method requires long computations that the solution above avoided, we note that this can be used to get an explicit solution to (\ref{eq:coclosedD}), which will therefore lead to explicit expressions of the $G_2$-structure and the metric.

\subsection{Extension on $\R^4 \times S^3$}\label{sec:extnoncomp}

In this section we consider the seven-dimensional non-compact simply connected cohomogeneity one manifold $M=\R^4 \times S^3$ with group diagram $\text{SU}(2)^2 \supset \Delta \text{SU}(2) \supset \{ 1 \}$. From the last section, we know that on the principal part of that manifold there is a family of coclosed $G_2$-structures. 
They are  constructed from a half-flat $\text{SU}(3)$-structure which is invariant under the cohomogeneity one action.
The next step is checking whether it is possible to extend this structure to the singular orbit.

The functions $A_i$ that, together with the constant $b_0 \neq 0$, give us our family of $G_2$-structures, satisfy the conditions (i) from Lemma \ref{lemma8} by construction.
Recall the expression of the $B_i$'s:
$$
B_i(t) = \text{sign}(b_0) \sqrt{\dfrac{D_j D_k}{D_i A_i^2}}, t \in (0,L), \quad B_i(0)=b_0,
$$
where $\{ i,j,k\}$ is a cyclic permutation of $\{ 1,2,3\}$.
Note that $B_i(t)$ is continuous at $t=0$.
We will check whether they satisfy the conditions (ii) from Lemma \ref{lemma8} for the metric to be extended to a singular orbit $Q = \text{SU}(2)^2/ \Delta \text{SU}(2) \cong S^3$.
First, a direct consequence from Lemma \ref{lemma:signDs} is that $B_i$'s are sign definite for $t>0$.
The next two Lemmas guarantee that, under the same boundary hypothesis for $A_1, A_2, A_3$, the remaining conditions from Lemma \ref{lemma8} hold.

\begin{lemma}\label{lemma:parity}
The functions $B_1, B_2, B_3$ from Corollary \ref{cor:exuniqsols} are even.
\end{lemma}
\begin{proof}
Recall that we can write (\ref{eq:coclosedD}) in matrix form as $dD/dt=M(t)D+N(t)$, where $D=(D_1, D_2, D_3)^T$ and both $M, N$ are matrices of odd functions. Let $D(t)$ be the unique solution of this equation in a neighbourhood of 0, say $[0, \varepsilon)$. 
We extend $M$ and $N$ to $(-\varepsilon, 0)$ by setting $M(-t) \coloneqq -M(t)$ and $N(-t) \coloneqq -N(t)$, $t \in (0, \varepsilon)$.
Then, since $D(t)$ is a solution of the equation 
$
dy(t)/d(t)=M(t)y(t)+N(t),
$
which is equivalent to 
$
dy(t)/d(-t)=M(-t)y(t)+N(-t),
$
we can therefore set $D(-t) \coloneqq D(t)$, which makes $D$ into an even function.
Finally, as the $D_i$'s are even, the $B_i$'s are even too.
\end{proof}

\begin{lemma}\label{lemma:dotdotB0}
The functions $B_1, B_2, B_3$ from Corollary \ref{cor:exuniqsols} satisfy $\ddot{B}_1(0)=\ddot{B}_2(0)=\ddot{B}_3(0)$.
\end{lemma}
\begin{proof}
We denote by $a_{i,3}$ the coefficient accompanying $t^3$ in the Taylor expansion of $A_i(t)$, $b_{i,2}$ the coefficient accompanying $t^2$ in the Taylor expansion of $B_i(t)$ and $d_{i,4}$ the parameter accompanying $t^4$ in the Taylor expansion of $D_i(t)$. We can then write:
$$
\begin{array}{ll}
A_i(t)=\dfrac{t}{2}+ a_{i,3}t^3 + O(t^5), \\
B_i(t)=b_0 + b_{i,2} t^2 + O(t^4), \\
D_i(t)= \dfrac{b_0^2}{4} t^2 +d_{i,4}t^4 + O(t^6).
\end{array}
$$
The parameters $d_{i,4}$ depend on $a_{i,3}$, and we can compute this dependence from the ODEs for $D(t)$ by considering the Taylor expansion of both sides of
$$
A_1 A_2 A_3 \dot{D}_i =(A_1 A_2 A_3)^2 -A_i^2 D_i + A_j^2 D_j + A_k^2 D_k,
$$
where $\{ i,j,k\}$ is a cyclic permutation of $\{ 1,2,3\}$.
After a direct computation we get that
\begin{equation*}
d_{i,4}=\dfrac{1}{16}-\dfrac{1}{2}b_0^2 a_{i,3}.
\end{equation*}
Now, as
$$
D_i(t) =A_j(t) B_j(t) A_k(t) B_k(t),
$$
we can also write
$$
d_{i,4}=\dfrac{b_0^2}{2} (a_{j,3}+a_{k,3}) + \dfrac{b_0}{4}(b_{j,2}+b_{k,2}).
$$
Putting these two together, we get
$$
b_{1,2}=b_{2,2}=b_{3,2}=\dfrac{1}{8b_0} -b_0(a_{1,3}+a_{2,3}+a_{3,3}).
$$
Hence, the condition $\ddot{B}_1(0)=\ddot{B}_2(0)=\ddot{B}_3(0)$ is always true.
\end{proof}

We can now write the following Proposition.

\begin{prop}\label{prop:coclosednoncompact}
Let $M=\R^4 \times S^3$ be the seven-dimensional non-compact simply connected cohomogeneity one manifold with group diagram $\text{SU}(2)^2 \supset \Delta \text{SU}(2) \supset \{ 1 \}$, with a $G_2$-structure constructed from a half-flat $\text{SU}(3)$-structure which is invariant under the cohomogeneity one action.
Let the $\text{SU}(3)$-structure $(\omega, \Omega_1, \Omega_2)$ be written as in (\ref{eq:SU(3)structure}).
Let $A_1, A_2, A_3:[0, \infty) \rightarrow \R$ be smooth functions with $A_i(t) >0$ for $t \in (0,\infty)$ such that
\begin{enumerate}[label=(\roman*)]
    \item $A_i$'s are odd;
    \item $\dot{A}_i(0)=1/2$.
\end{enumerate}
Let $b_0 \neq 0$. Then there exist unique functions $B_1, B_2, B_3: [0,\infty) \rightarrow \R$ that, together with $A_1, A_2, A_3$, will give a solution to (\ref{eq:coclosedAB}) with $B_i(0)=b_0$, and the $B_i$'s depend continuously on $b_0 \neq 0$.
The metric $g$ given by (\ref{eq:g}) extends smoothly over the singular orbit.
\end{prop}
\begin{proof}
By Corollary \ref{cor:exuniqsols}, there exists unique functions $B_1, B_2, B_3: [0,\infty) \rightarrow \R$ that, together with $A_1, A_2, A_3$, will give a solution to (\ref{eq:coclosedAB}) with $B_i(0)=b_0$, and the $B_i$'s depend continuously on $b_0 \neq 0$.
Note that the $B_i$'s are smooth functions in $[0,\infty)$, and $B_1(0)=B_2(0)=B_3(0) \neq 0$.
By Lemma \ref{lemma:parity}, the functions $B_i$ are even, and by Lemma \ref{lemma:dotdotB0}, $\ddot{B}_1(0)=\ddot{B}_2(0)=\ddot{B}_3(0)$. They are also sign definite. Hence, by Lemma \ref{lemma8}, as the conditions for the functions $A_i$ are satisfied by construction, the metric $g$ extends smoothly over the singular orbit.
\end{proof}

The next Theorem summarizes the results from the previous Proposition.

\begin{theorem}
On the cohomogeneity one manifold $M=\R^4 \times S^3$ with group diagram $\text{SU}(2)^2 \supset \Delta \text{SU}(2) \supset \{ 1 \}$, there is a family of $\text{SU}(2)^2$-invariant coclosed $G_2$-structures which is given by three positive smooth functions $A_1, A_2, A_3: [0, \infty) \rightarrow \R$
satisfying the boundary conditions at $t=0$
$$
A_i(t)= \dfrac{t}{2} + O(t^3),
$$
and a non-zero parameter.
Moreover, any $\text{SU}(2)^2$-invariant coclosed $G_2$-structure constructed from a half flat $\text{SU}(3)$-structure is in this family.
\end{theorem}

\begin{remark}
The volume of the singular orbit at $t=0$ is proportional to $b_0^3$.
\end{remark}

\subsection{Extension on $S^4 \times S^3$}\label{sec:extcomp}

In this section we consider the seven-dimensional compact simply connected cohomogeneity one manifold $M=S^4 \times S^3$ with group diagram $\text{SU}(2)^2 \supset \Delta \text{SU}(2), \Delta \text{SU}(2) \supset \{ 1 \}$.
From Section \ref{sec:existence}, we know that on the principal part of that manifold there is a family of coclosed $G_2$-structures constructed from a half-flat $\text{SU}(3)$-structure. In the last section, we saw that it is possible to extend the structure to one singular orbit $\text{SU}(2)^2/ \Delta \text{SU}(2)$.

The next Proposition shows that the previous structure cannot be smoothly extended to two singular orbits of type $\text{SU}(2)^2/ \Delta \text{SU}(2)$.

\begin{prop}
On the cohomogeneity one manifold $M=S^4 \times S^3$ with group diagram $\text{SU}(2)^2 \supset \Delta \text{SU}(2), \Delta \text{SU}(2) \supset \{ 1 \}$, there are no $\text{SU}(2)^2$-invariant coclosed $G_2$-structures constructed from half-flat $\text{SU}(3)$-structures.
\end{prop}
\begin{proof}
Suppose that $M$ has a $G_2$-structure constructed from a half-flat $\text{SU}(3)$-structure on its principal part, which is invariant under the cohomogeneity one action. 
Then the $\text{SU}(3)$-structure $(\omega, \Omega_1, \Omega_2)$ can be written as in (\ref{eq:SU(3)structure}), and there are
$A_1, A_2, A_3:[0, 1] \rightarrow \R$ smooth functions with $A_i(t) >0$ for $t \in (0,1)$ such that
\begin{enumerate}[label=(\roman*)]
    \item $A_i$'s are odd around $t=0$;
    \item $\dot{A}_i(0)=1/2$.
\end{enumerate}
Let $b_0 \neq 0$.
By Proposition \ref{prop:exuniqsols} we know that there exists a unique solution of the system (\ref{eq:coclosedD}), $D_i: [0,1) \rightarrow \R$ with initial conditions $D_i(0)=\dot{D}_i(0)=0, \ddot{D}_i(0)=b_0^2/4$, for $i=1,2,3$. 
Then, there exist unique functions $B_1, B_2, B_3: [0,1] \rightarrow \R$ that, together with $A_1, A_2, A_3$, will solve (\ref{eq:coclosedAB}).
Also, by Proposition \ref{prop:coclosednoncompact} the metric (\ref{eq:g}) can be extended to the singular orbit at $t=0$. Hence, this solution can be extended to the interval $[0,1)$.
It remains to check whether we can extend the solutions to $t=1$. 
In equation (\ref{eq:coclosedD}), we add 
$$
\dfrac{d}{dt} (D_1 + D_2 + D_3) = 3 A_1 A_2 A_3 + \dfrac{A_1^2 D_1 + A_2^2 D_2 + A_3^2 D_3}{A_1 A_2 A_3},
$$
In Lemma \ref{lemma:signDs} we proved that for $t \in (0,1)$, $D_i(t) >0$, $i=1,2,3$ so $d (D_1 + D_2 + D_3)/dt >0$. In particular, as $D_1+D_2+D_3$ vanishes at $t=0$, it cannot be 0 at $t=1$. 
Therefore, there cannot be an extension to the singular orbit at $t=1$.
This is because the $B_i$ functions obtained from the $D_i$'s blow up at this orbit.\\
Alternatively, this result also follows from the smoothness conditions for $A_i$, $i=1,2,3$ around both singular orbits. From the symmetry of equations (\ref{eq:coclosedAB}), we deduce that the smoothness conditions at $t=1$ require that $A_i(1)=0$, $\dot{A}_i(1)=1/2$. This, together with $A_i(0)=0$, $\dot{A}_i(0)=1/2$, implies that there is a point in $(0,1)$ where the $A_i$ vanishes. This gives a contradiction with the fact that $A_i$ is sign definite in $(0,1)$. 
\end{proof}

\subsection{Conclusions and future directions}\label{sec:conclusions}

The first thing that we observe is that on $M= \R^4 \times S^3$ there are many more $G_2$-structures constructed from half-flat $\text{SU}(3)$-structures that are coclosed than that are torsion-free (in general, one can perturb a coclosed $G_2$-structure by adding a small exact 4-form to construct another coclosed $G_2$-structure). 
In particular, Lotay and Oliveira give a two-parameter family of torsion-free $G_2$-structures constructed from half-flat $\text{SU}(3)$-structures \cite[Remark 21]{Lotay17}. The coclosed family described in this chapter depends on three smooth functions satisfying certain boundary conditions and a non-zero parameter.
More generally, known examples of torsion-free $G_2$-structures constructed from half-flat $\text{SU}(3)$-structures have an extra $\text{U}(1)$-symmetry.
With our previous notation, this means that $A_2=A_3$ and $B_2 =B_3$. We showed that if we relax the torsion-free condition to coclosed, the $A_i$'s only need to be equal at orders lower or equal than 1 in the Taylor expansion of the cohomogeneity one parameter around the singular orbit, and the $B_i$'s at orders lower or equal than 2. 
Our families of structures contain the Bryant--Salamon $G_2$-holonomy metric \cite{BS89} and the 1-parameter family of complete $\left( \text{SU}(2)^2 \times \text{U}(1) \right)$-invariant $G_2$-metrics of Brandhuber et al.\ \cite{Brandhuber01} and Bogoyavlenskaya \cite{Bogoyavlenskaya13}, also known as the $\mathbb{B}_7$ family. 
In particular, we deduce from \cite[Theorem 6.16]{FHN21} that if we have an extra $\text{U}(1)$-symmetry and impose that the metric is $G_2$ and complete, we get precisely this family.

While the objective of this paper is the construction of new examples of coclosed $G_2$-structures which are invariant under the cohomogeneity one action of $\text{SU}(2)^2$, we outline future directions for a systematic study of these structures which could be made by relaxing some of the assumptions made in this paper. We first recall that coclosed $G_2$-structures do not need to be half-flat. While this condition is guaranteed to hold for torsion-free $G_2$-structures, when our structures are coclosed we only need a weaker condition: $d\omega^2=0$.
Therefore, we may consider families of $\text{SU}(3)$-structures $(\omega(t), \Omega_2(t))$ such that $d\omega^2=0$ instead of being half-flat, but this would require a different analysis as the description of the structure from equation (\ref{eq:SU(3)structure}) would no longer hold. 
It would also be interesting to consider seven-dimensional manifolds with a $G_2$-structure invariant under a cohomogeneity one action of $\text{SU}(2)^2$ which is different than the one considered in this paper and described in Section \ref{sec:extension} (see Remark \ref{rmk:inequivactions} for other inequivalent actions).
As the coclosed equations (\ref{eq:coclosedD}) do not depend on the action, these will remain valid for a different action of $\text{SU}(2)^2$, while the conditions for a smooth extension to the singular orbit will be different.
A similar analysis of the ODEs with the new boundary conditions could be done for these cases, potentially providing new families of metrics containing $\mathbb{D}_7$ family and the infinite extension of the $\mathbb{C}_7$ family as particular cases.

In \cite{AS22}, A. and Salvatore showed that if $M$ is a six-dimensional simply connected cohomogeneity one manifold under the almost effective action of a connected Lie group G, then $M$ admits no $G$-invariant balanced non-Kähler $\text{SU}(3)$-structures.
As in this paper, the search for balanced $\text{SU}(3)$-structures was also motivated by heterotic string theory; in particular, the Hull--Strominger system.
This means that on most cases (and possibly always), the existence of a balanced structure (which can be seen as the six-dimensional analogue to a coclosed $G_2$-structure) in the cohomogeneity one setting forces the structure to be K\"ahler. 
We observe that the existence of a class of coclosed $G_2$-structures, not necessarily torsion-free, in the cohomogeneity one setting contrasts with this result, as in the 7 dimensional analogue, the cohomogeneity one hypothesis does not force coclosed $G_2$-structures to be torsion-free.

Given the structures found in this paper, the question that arises is what are the $G_2$-instantons over $\R^4 \times S^3$ with these structures. Another question is whether it is possible to find solutions to the heterotic $G_2$ system over them. At the present time, the author is working on these problems (see Remark \ref{rmk:flux}).

\begin{footnotesize}
\renewcommand{\refname}{References}
\bibliographystyle{alpha}
\bibliography{biblio}
\end{footnotesize}

\end{document}